\documentclass{amsart}
\usepackage{amssymb,amsmath,amsthm,tikz}
\usetikzlibrary{tikzmark,calc}

\usepackage{graphicx}
\usepackage{xcolor}

\usepackage[normalem]{ulem}

\theoremstyle{theorem}
\newtheorem{theorem}{Theorem}[section]
\newtheorem{prop}[theorem]{Proposition}
\newtheorem{corollary}[theorem]{Corollary}

\theoremstyle{definition}
\newtheorem{definition}[theorem]{Definition}

\theoremstyle{example}
\newtheorem{example}[theorem]{Example}

\newcommand{\boundellipse}[3]
{(#1) ellipse (#2 and #3)
}

\DeclareMathOperator{\mex}{mex}

\allowdisplaybreaks

\title[{Overpartitions and Mex Sequences}]{Overpartitions and Kaur, Rana, and Eyyunni's Mex Sequences}

\author{Brian Hopkins}
\address{Department of Mathematics and Statistics, 
Saint Peter's University, Jersey City, NJ 07306, USA}
\email{bhopkins@saintpeters.edu}

\author{James A. Sellers}
\address{Department of Mathematics and Statistics, 
University of Minnesota Duluth, Duluth, MN 55812, USA}
\email{jsellers@d.umn.edu}



\date{}

\begin{document}

\maketitle

\begin{abstract}
Kaur, Rana, and Eyyunni recently defined the mex sequence of a partition and established, by analytic methods, connections to two disparate types of partition-related objects.  We make a bijection between partitions with certain mex sequences and a uniform family of overpartitions which allows us to provide combinatorial proofs of their results, as they requested.
\end{abstract}

\section{Introduction and background}

A partition of a positive integer $n$ is a finite sequence of positive integers $\lambda = (\lambda_1, \ldots, \lambda_j)$ with $\sum \lambda_i = n$, also written $|\lambda| = n$.  
The $\lambda_i$, called the parts of $\lambda$, are ordered so that $\lambda_1 \ge \cdots \ge \lambda_j$.  We write $P(n)$ for the set of all partitions of $n$ and $p(n) = \vert P(n) \vert$ with the convention that $p(0) =1$.   For example,
\[P(4) = \{(4), (3,1), (2,2), (2,1,1), (1,1,1,1)\}\]
and $p(4) = 5$.  

Recently, much attention has been paid to the mex of a partition, a simple partition statistic whose name is a portmanteau of minimal excludant.  

\begin{definition}
Given a partition $\lambda$, the mex of $\lambda$ is the least positive integer that is not a part of $\lambda$.
\end{definition}

This term was first applied to partitions in 2019 by Andrews and Newman \cite{an} but the concept was considered by 2006 at the latest by Grabner and Knopfmacher \cite{gk}.  Several authors have studied the mex statistic from a wide variety of perspectives; see, for instance, \cite{hs20, hs24, hss, hsy, k, y}.

Among the many generalizations of $\mex(\lambda)$ that have been considered, here we focus on the mex sequence of a partition defined in 2024 by Kaur, Rana, and Eyyunni \cite{kre}.

\begin{definition}
Given a partition $\lambda$, the mex sequence of $\lambda$ is the longest sequence of consecutive missing integers in the partition starting with $\mex(\lambda)$.
\end{definition}

Note that this is a different set-valued generalization of the mex than those considered by Knopfmacher and Warlimont \cite{kw} and Bhoria, Eyyunni, and Li \cite{bel}.  Several examples of mex and mex sequences follow.

\begin{example}
(a) The partition $(4)$ has mex  $1$ and mex sequence $(1,2,3)$.  \\
(b) The partition $(9,4,4,3,1)$ has mex $2$ and mex sequence $(2)$.  \\
(c) The partition $(4,3,3,3,2,1,1)$ has mex is $5$ and mex sequence $(5,6,7,\dots)$.  
\end{example}

With this, Kaur, Rana, and Eyyunni \cite{kre} defined the following family of partitions.  

\begin{definition}
For any integer $r\geq 1$, let $P_r^{\textrm{mex}}(n)$ be the set of partitions of $n$ whose mex sequences have length at least $r$ with
$p_r^{\textrm{mex}}(n) = \vert P_r^{\textrm{mex}}(n) \vert$.
\end{definition}

For all $r\geq 1$ and any $n \ge 0$, it is clear that $P_{r+1}^{\textrm{mex}}(n) \subseteq P_{r}^{\textrm{mex}}(n)$.  Moreover, $P_1^{\textrm{mex}}(n) = P(n)$ for all $n$ since the mex sequence of any partition of $n$ has length at least 1 (the mex itself).

In \cite{kre}, the authors use elementary generating function manipulations, as well as a version of Heine's transformation of a ${}_2\phi_1 $ summation formula, to obtain the following generating function identity for $p_r^{\textrm{mex}}(n)$.  It incorporates the standard Pochhammer notation, i.e., for $n\geq 1,$ $(a;q)_n = (1-a)(1-aq)\cdots(1-aq^{n-1})$  and  $(a;q)_\infty = \lim_{n \rightarrow \infty} (a;q)_n$ where we assume that $|q|<1$.  

\begin{theorem}[Theorem 10, \cite{kre}]
\label{thm:kre_genfn}
For $r$ a positive integer, the generating function for $p_r^{\textrm{mex}}(n)$ is
\[\sum_{n \ge 0} p_r^{\textrm{mex}}(n) q^n = \frac{1}{(q;q^2)_\infty (q^{r+1};q^2)_\infty}.\]
\end{theorem}

Aiming for a combinatorial result, Kaur, Rana, and Eyyunni defined a set of restricted partitions and a particular version of two-colored partitions as follows.

\begin{definition}
Let $P_e^{>r}(n)$ be the set of partitions of $n$ in which no even integer less than $r$ is allowed to be a part with $p_e^{>r}(n) = \vert P_e^{>r}(n) \vert$.  

Let $P_{o,2}^{>r}(n)$ be the set of partitions of $n$ into odd parts where parts greater than $r$ come in two colors with $p_{o,2}^{>r}(n) = \vert P_{o,2}^{>r}(n) \vert$.
\end{definition}

For example, using subscripts for colors, $p_{o,2}^{>2}(6) = 8$ from
\[\{(5_1, 1), (5_2, 1), (3_1, 3_1), (3_1, 3_2), (3_2, 3_2), (3_1, 1, 1, 1), (3_2, 1, 1, 1), (1,1,1,1,1,1)\}.\]
Note that $(3_2, 3_1)$ is not a distinct partition in $P_{o,2}^{>2}(6)$ and the parts 1 can be considered as $1_1$, i.e., having the first color.

The following result then follows from elementary interpretations of the generating function identity in Theorem \ref{thm:kre_genfn}.

\begin{corollary}[Corollary 11, \cite{kre}] \label{needcomb}
\label{cor:kre_comb_interpretation}  For each integer $r \ge 1$,
\[p_r^{\textrm{mex}}(n) = \begin{cases} p_e^{>r}(n) & \text{if $r$ is odd,} \\ p_{o,2}^{>r}(n) & \text{if $r$ is even.}\end{cases} \]
\end{corollary}

In their concluding remarks, the authors write \cite[p.\ 1129]{kre}, ``It would be highly desirable to get a bijective proof of the identities for $p_r^{\textrm{mex}}(n)$ in Corollary 11.''  We will provide such proofs after establishing a connection between the mex sequence and an important partition generalization known as overpartitions.

An overpartition of a positive integer $n$ is a partition of $n$ wherein the first occurrence of a part may be overlined.  That is, any overlined parts must be distinct while there is no restriction on non-overlined parts.  We write $\overline{P}(n)$ for the set of of overpartitions of $n$ and $\overline{p}(n) = \vert \overline{P}(n) \vert$ with the convention that $\overline{p}(0) =1$.
For example, 
\begin{gather*}
\overline{P}(4) = \{ (4), (\overline{4}), (3, 1), (3, \overline{1}), (\overline{3}, 1), (\overline{3}, \overline{1}), (2, 2), (\overline{2},2), \\
(2, 1, 1), (2, \overline{1}, 1), (\overline{2}, 1, 1), (\overline{2}, \overline{1}, 1), (1, 1, 1, 1), (\overline{1}, 1, 1, 1)\}
\end{gather*}
and $\overline{p}(4) = 14$.  We will sometimes write an overpartition as $(\lambda, \mu)$ where $\lambda$ is a partition with distinct parts (corresponding to the overlined parts) and $|\lambda| + |\mu| = n$ so that, for example, $(\overline{2}, \overline{1}, 1)$ corresponds to $((2,1),(1))$.
Overpartitions were named by Corteel and Lovejoy in 2004 \cite{cl} although, as they discuss, equivalent sets of partitions had previously
been considered in various contexts.

It will be helpful to recall the generating functions for the number of partitions and overpartitions of $n$:
\[ \sum_{n \ge 0} p(n) q^n = \frac{1}{(q;q)_\infty}, \quad \sum_{n \ge 0} \overline{p}(n) q^n = \frac{(-q;q)_\infty}{(q;q)_\infty}.\]

We conclude this section by reviewing two classical maps on partitions from the 19th century detailed in \cite{ae}.

Using Ferrers's graphical representation of a partition, where each part corresponds to a row of dots, an obvious operation is conjugation which swaps rows and columns.  This map $\lambda \mapsto \lambda'$ is clearly an involution.  One consequence is the equal count of two types of restricted partitions.

\begin{prop} \label{nogaps}
The number of partitions of $n$ into distinct parts equals the number of partitions of $n$ where each part size from 1 to $\lambda_1$ appears at least once.
\end{prop}

Although not stated directly, this follows from a discussion in Sylvester and Franklin's major 1882 work \cite[pp.\ 272--274]{sf}.  Note that partitions where each part size from 1 to $\lambda_1$ appears at least once, also called partitions with no gaps,  have mex $\lambda_1 + 1$ and an infinite mex sequence.
See Table \ref{conj} for the connection between such partitions of 6.    In such examples, we use a shorthand notation for partitions including superscripts for repetition, e.g., $21^4$ for $(2,1,1,1,1)$.

\begin{table}[h]
\renewcommand{\arraystretch}{1.5}
\begin{tabular}{r|cccc}
$P(6)$ distinct parts & 6 & 51 & 42 & 321 \\ \hline
$P(6)$ no gaps & $1^6$ & $21^4$ & $2211$ & $321$
\end{tabular}
\medskip
\caption{The $n = 6$ example of the types of partitions in Proposition \ref{nogaps}, connected by conjugation.}  \label{conj}
\end{table}

In one of the first published results on partitions, Euler used generating functions to show that the number of partitions of $n$ into distinct parts equals the number of partitions of $n$ into parts all odd.  Using Pochhammer notation, this is simply encoded by the generating function identity
\begin{equation}
(-q;q)_\infty = \frac{1}{(q;q^2)_\infty}. \label{edo}
\end{equation}

Glaisher provided a combinatorial proof of \eqref{edo} in 1883 as part of a generalization of Euler's result \cite{g}.  It is easily described as follows:  In a partition which contains distinct parts, the even parts are split into halves until all parts in the newly-obtained partition are odd. Conversely, in a partition with odd parts, any two repeated parts are merged into a single part (of twice the size) until all parts in the new partition are distinct.  Clearly, this operation produces an involution.  

See Table \ref{gl6} for the connection between such partitions of 6.

\begin{table}[h]
\renewcommand{\arraystretch}{1.5}
\begin{tabular}{r|cccc}
$P(6)$ distinct parts & 6 & 51 & 42 & 321 \\ \hline
$P(6)$ odd parts & 33 & 51 & $1^6$ & $31^3$
\end{tabular}
\medskip
\caption{Glaisher's map for $n = 6$.}  \label{gl6}
\end{table}

In the next section, we establish a bijection between $P_r^{\textrm{mex}}(n)$ and a certain family of overpartitions.  With this is hand, we satisfy Kaur, Rana, and Eyyunni's request for combinatorial proofs of Corollary \ref{needcomb}.

\section{Overpartitions and bijections}

We begin with an unexpected connection between certain restricted overpartitions and regular partitions.  

Clearly, overpartitions with no overlined parts are simply partitions.  We show that the number of overpartitions of $n$ where any non-overlined parts are even also equals $p(n)$.

\begin{prop} \label{r1}
There are $p(n)$ overpartitions of $n$ where any non-overlined parts are even, i.e., overpartitions $(\lambda, \mu)$ where $\mu$ consists of only even parts.
\end{prop}

\begin{proof}
We establish a bijection between $P(n)$ and the described subset of $\overline{P}(n)$.

Given a partition $\kappa$ of $n$, apply Glaisher's map to any odd parts and overline the resulting distinct parts while leaving any even parts of $\kappa$ non-overlined.  This produces an overpartition of $n$ where any non-overlined parts are even.

Given an overpartition $(\lambda,\mu)$ of $n$ where all parts of $\mu$ are even, apply Glaisher's map to $\lambda$ and take the union of the resulting odd parts and $\mu$ to make an ordinary partition of $n$.

Because Glaisher's map establishes a bijection, the maps described here are inverses.
\end{proof}

See Table \ref{r1ex} for an example of the maps in the proof of Proposition \ref{r1}.

\begin{table}[h]
\renewcommand{\arraystretch}{1.5}
\begin{tabular}{r|ccccccccccc}
$P(6)$ & 6 & 51 & 42 & 411 & 33 & 321 & $31^3$ & $2^3$ & 2211 & $21^4$ & $1^6$ \\ \hline
overpartitions & 6 & $\overline{51}$ & 42 & $4\overline{2}$ & $\overline{6}$ & $\overline{3}2\overline{1}$ & $\overline{321}$ & $2^3$ & $\overline{2}22$ & $\overline{4}2$ & $\overline{42}$ 
\end{tabular}
\medskip
\caption{The $n = 6$ example of the correspondence of Proposition \ref{r1}.}  \label{r1ex}
\end{table}

Recall the observation that $P_1^{\textrm{mex}}(n) = P(n)$.  Proposition \ref{r1} then suggests a connection between partitions with longer mex sequences and more restricted overpartitions.

\begin{definition} \label{opr}
For any integer $r \ge 1$, let $\overline{P}_r(n)$ be the set of overpartitions of $n$ where overlined parts have no restrictions (besides being distinct) and non-overlined parts must be greater than $r$ and have the same parity as $r+1$, and let $\overline{p}_r(n) = \vert \overline{P}_r(n) \vert$.  
\end{definition}

The tables after the following results include several examples of $\overline{P}_r(n)$ for various $r$ and $n$.  

With this notation, Theorem \ref{r1} establishes $\overline{p}_1(n) = p_1^{\textrm{mex}}(n)$ for all $n$.  This is the first case of a general result.  

\begin{theorem} \label{mexseqover}
For any integer $r \ge 1$, $\overline{p}_r(n) = p_r^{\textrm{mex}}(n)$ for all $n$.
\end{theorem}

We provide both a generating function proof and, in the spirit of the remaining results, a combinatorial proof of this connection.  The latter will use a part-wise combination of partitions explained by the example $(4,2) \oplus (5,1,1) = (9,3,1)$.

\begin{proof}
Combining Theorem \ref{thm:kre_genfn} with Euler's identity \eqref{edo} gives
\[ \sum_{n\ge0} p_r^{\textrm{mex}}(n) q^n = \frac{1}{(q;q^2)_\infty (q^{r+1};q^2)_\infty} = \frac{(-q;q)_\infty}{(q^{r+1};q^2)_\infty}\]
which is the generating function for the overpartitions $\overline{P}_r(n)$ (with the restriction on non-overlined parts apparent in the denominator).

For the combinatorial proof, we establish a bijection between $\overline{P}_r(n)$ and $P_r^{\textrm{mex}}(n)$.

Given $\kappa \in P_r^{\textrm{mex}}(n)$, if the mex sequence is infinite, then $\kappa$ has no gaps and $\kappa'$ is a distinct part partition by Proposition \ref{nogaps}; send $\kappa$ to the overpartition $(\kappa', \varnothing)$ which, with only overlined parts, is trivially in $\overline{P}_r(n)$.  Otherwise, write
\[\kappa = (\kappa_1, \ldots, \kappa_i, \sigma_{i+1}, \ldots, \sigma_j)\]
where $\kappa_i > \mex(\kappa)$ and $\sigma_{i+1} = \mex(\kappa)-1$ (possibly 0).  We describe how to write $\kappa$ as
\[\kappa = (\kappa_1 - \sigma_1, \ldots, \kappa_i - \sigma_i) \oplus (\sigma_1, \ldots, \sigma_j)\]
where $(\kappa_1 - \sigma_1, \ldots, \kappa_i - \sigma_i)$ is a partition with each part at least $r$ and $(\sigma_1, \ldots, \sigma_j)$ is a partition with no gaps.

Since $\kappa \in P_r^{\textrm{mex}}(n)$, in fact $\kappa_i \ge \sigma_{i+1}+r+1$.  Working from $\ell = i$ to $\ell = 1$, let
\[\sigma_\ell = \begin{cases} \sigma_{\ell + 1} & \text{if $\kappa_\ell - \sigma_{\ell+1} \equiv r \bmod{2}$,} \\ \sigma_{\ell+1} + 1 & \text{if $\kappa_\ell - \sigma_{\ell+1} \not\equiv r \bmod{2}$.} \end{cases} \]
Note that $\kappa_i - \sigma_i \ge \sigma_{i+1}+r+1 - (\sigma_{i+1}+1) = r$, as desired.  Further, if $\kappa_{i-1} = \kappa_i$, then $\sigma_{i-1} = \sigma_i$ gives $\kappa_{i-1} - \sigma_{i-1}$ the necessary parity and $\kappa_{i-1} - \sigma_{i-1} = \kappa_i - \sigma_i \ge r$.  Otherwise, $\kappa_{i-1} \ge \kappa_i + 1$ and 
\[\kappa_{i-1} - \sigma_{i-1} \ge \kappa_i +1 - (\sigma_i+1) = \kappa_i - \sigma_i \ge r. \]
Continuing in this way, we see that $(\kappa_1 - \sigma_1, \ldots, \kappa_i - \sigma_i)$ is a partition with each part at least $r$.  Since $(\sigma_{i+1},\ldots,\sigma_j)$ is a partition with no gaps and each $\sigma_\ell$ satisfies $\sigma_{\ell+1} \le \sigma_\ell \le \sigma_{\ell+1}+1$, the resulting $(\sigma_1,\ldots,\sigma_j)$ is also a partition with no gaps so that, by Proposition \ref{nogaps}, its conjugate has distinct parts.  Assign $\kappa$ to the overpartition
\[ ((\sigma_1,\ldots,\sigma_j)', (\kappa_1 - \sigma_1, \ldots, \kappa_i - \sigma_i)).\]

Given $(\lambda, \mu) \in \overline{P}_r(n)$, send the overpartition to $\kappa = \lambda' \oplus \mu \in P(n)$.  We need to show that $\kappa$ has mex sequence of length at least $r$.
Writing $\lambda = (\lambda_1, \ldots, \lambda_j)$, by Proposition \ref{nogaps}, the conjugate $\lambda'$ has first part $j$, length $\lambda_1$, and no gaps.  Suppose $\mu$ consists of $m$ parts (which are all greater than $r$).  

If $m=0$, then $\kappa = \lambda'$ has mex $j+1$ and an infinite mex sequence.    

If $1 \le m < \lambda_1$, then $\kappa_m > \lambda'_m + r$ and $\kappa_{m+1} = \lambda'_{m+1}$ which is either $\lambda_m$ or $\lambda_m - 1$ and every smaller part size appears at least once since $\lambda'$ has no gaps.  Therefore $\kappa$ has mex sequence at least $r$.  

Finally, if $m \ge \lambda_1$, then $\kappa = \lambda' \oplus \mu$ has each part greater than $r$ and mex sequence at least $r$.

In each case, $\kappa \in P_r^{\textrm{mex}}(n)$.

It is straightforward to confirm that each map is injective, so that together the maps establish a bijection, completing the combinatorial proof.
\end{proof}

Several examples of the maps in the combinatorial proof of Theorem \ref{mexseqover} follow, along with Table \ref{rex}.

\begin{example}
(a) The partition $\kappa = (8,7,3,2,1,1)$ has mex sequence $(4,5,6)$.  For $r = 2$, the non-overlined parts of overpartitions in $\overline{P}_2(22)$ need to be odd and greater than 2.  Therefore $\sigma_2 = 4$ giving $7 - 4 = 3$ and $\sigma_1 = 5$ giving $8-5 = 3$ so that 
\[(8,7,3,2,1,1) = (3,3) \oplus (5,4,3,2,1,1) \mapsto (\overline{6}, \overline{4}, \overline{3}, 3, 3, \overline{2}, \overline{1}).\]
(b) For $r=3$, the same partition $\kappa$ has $\sigma_2 = 3$ giving $7-3=4$ and $\sigma_1 = 4$ giving $8-4=4$.  Therefore,
\[(8,7,3,2,1,1) = (4,4) \oplus (4,3,3,2,1,1) \mapsto  (\overline{6}, \overline{4}, 4, 4, \overline{3}, \overline{1}) \in \overline{P}_3(22).\]
(c) For $r = 2$, the overpartition $(\overline{5},5,\overline{3},3,3,\overline{2},\overline{1}) \in \overline{P}_2(22)$ has $\lambda = (5,3,2,1)$ with conjugate $\lambda' = (4,3,2,1,1)$ and $\mu = (5,3,3)$ so that it maps to
\[ (4,3,2,1,1) \oplus (5,3,3) = (9,6,5,1,1) \in P(22)\]
which has mex sequence $(2,3,4)$ and is thus an element of $P_2^{\textrm{mex}}(22)$.
\end{example}

\begin{table}[h]
\renewcommand{\arraystretch}{1.5}
\begin{tabular}{r|cccccccccc}
$P_2^{\textrm{mex}}(7)$ & 7 & 61 & 511 & 43 & $41^3$ & 3211 & $2^31$ & $221^3$ & $21^5$ & $1^7$ \\ \hline
$\overline{P}_2(7)$ & 7 & $5\overline{2}$ & $\overline{3}3\overline{1}$ & $33\overline{1}$ & $\overline{4}3$ & $\overline{421}$ & $\overline{43}$ & $\overline{52}$ & $\overline{61}$ & $\overline{7}$ \\ \hline \hline
$\overline{P}_3(7)$ & $\overline{7}$ & $6\overline{1}$ & $\overline{61}$ & $\overline{52}$ & $4\overline{3}$ & $\overline{43}$ & $4\overline{21}$ & $\overline{421}$ \\ \hline
$P_3^{\textrm{mex}}(7)$ & $1^7$ & 7 & $21^5$ & $221^3$ & 511 & $2^31$ & 61 & 3211 
\end{tabular}
\medskip
\caption{The $n = 7$ cases of the correspondence of Theorem \ref{mexseqover} for $r = 2$ and $r = 3$.}  \label{rex}
\end{table}

With this unified interpretation of $p_r^{\textrm{mex}}(n)$ in terms of overpartitions, we can provide combinatorial proofs of Corollary \ref{needcomb} requested by Kaur, Rana, and Eyyunni.  Given the very different nature of their combinatorial interpretations depending on the parity of $r$, we give two bijections.

\begin{prop} \label{rodd}
For each odd $r \ge 1$, $\overline{p}_r(n) = p_e^{>r}(n)$ for all $n$.
\end{prop}

The $r=1$ case is Proposition \ref{r1}.  In fact, as detailed next, the result for any odd $r$ follows from the same map applied to the appropriate subsets.

\begin{proof}
Given $r \ge 1$ odd, we establish a bijection between $P_e^{>r}(n)$ and $\overline{P}_r(n)$.

Given a partition $\kappa \in P_e^{>r}(n)$, apply Glaisher's map to any odd parts and overline the resulting distinct parts while leaving any even parts of $\kappa$ (necessarily greater than $r$) non-overlined.  This produces an overpartition of $n$ where any non-ovelined parts are even and greater than $r$, i.e., an element of $\overline{P}_r(n)$.

Given an overpartition $(\lambda, \mu) \in \overline{P}_r(n)$, note that $\mu$ consists entirely of even parts greater than $r$.  Apply Glaisher's map to $\lambda$ and take the union of the resulting odd parts and $\mu$ to make a partition in $P_e^{>r}(n)$.

As explained in the proof of Proposition \ref{r1}, these maps are inverses.
\end{proof}

Table \ref{r1ex} includes examples of the maps in the proof of Proposition \ref{rodd} since $P_e^{>3}(6) \subset P(6)$ and $P_e^{>5}(6) \subset P(6)$.  Table \ref{r3ex} gives another example.

\begin{table}[h]
\renewcommand{\arraystretch}{1.5}
\begin{tabular}{r|ccccccccccc}
$P_e^{>3}(8)$ & 8 & 71 & 611 & 53 & $51^3$ & 44 & 431 & $41^4$ & 3311 & $31^5$ & $1^8$  \\ \hline
$\overline{P}_3(8)$ & 8 & $\overline{71}$ & $6\overline{2}$ & $\overline{53}$ & $\overline{521}$ & 44 & $4\overline{31}$ & $\overline{4}4$ & $\overline{62}$ & $\overline{431}$ & $\overline{8}$ 
\end{tabular}
\medskip
\caption{The $n = 8$, $r = 3$ example of the correspondence of Proposition \ref{rodd} which includes the $r = 5$ and $r = 7$ cases.}  \label{r3ex}
\end{table}

Combined with Theorem \ref{mexseqover}, Proposition \ref{rodd} gives a combinatorial proof of Corollary \ref{needcomb} for the odd $r$ case, as desired.

Finally, we proceed to the even $r$ case.

\begin{prop} \label{reven}
For each even $r \ge 2$, $\overline{p}_r(n) = p_{o,2}^{>r}(n)$ for all $n$.
\end{prop}

\begin{proof}
Given $r \ge 2$ even, we establish a bijection between $P_{o,2}^{>r}(n)$ and $\overline{P}_r(n)$.

Given $\nu \in P_{o,2}^{>r}(n)$, write $\nu = (\pi, \rho)$ where $\pi$ consists of any odd parts of the first color (including any odd parts less than $r$) and $\rho$ consists of any odd parts of the second color (each at least $r$).  Apply Glaisher's map to $\pi$, overline the resulting distinct parts, and leave any parts of $\rho$ non-overlined.  This produces an overpartition of $n$ where any non-overlined parts are odd and greater than $r$, i.e., an element of $\overline{P}_r(n)$.

Given an overpartition $(\lambda, \mu) \in \overline{P}_r(n)$, note that $\mu$ consists entirely of odd parts greater than $r$.  Apply Glaisher's map to $\lambda$, assign the resulting odd parts the first color, and assign the parts of $\mu$ the second color.  The union of these colored odd parts is an element of $P_{o,2}^{>r}(n)$.

Since Glaisher's map is a bijection, it is clear that these maps are inverses.
\end{proof}

See Table \ref{r2ex} for an example of the maps in the proof of Proposition \ref{reven}.

\begin{table}[h]
\renewcommand{\arraystretch}{1.5}
\begin{tabular}{r|cccccccc}
$P_{o,2}^{>2}(6)$ & $5_11_1$ & $5_21_1$ & $3_13_1$ & $3_13_2$ & $3_23_2$ & $3_1(1_1)^3$ & $3_2(1_1)^3$ & $(1_1)^6$  \\ \hline
$\overline{P}_2(6)$ & $\overline{51}$ & $5\overline{1}$ & $\overline{6}$ & $\overline{3}3$ & 33 & $\overline{321}$ & $3\overline{21}$ & $\overline{42}$ 
\end{tabular}
\medskip
\caption{The $n = 6$, $r = 2$ example of the correspondence of Proposition \ref{reven}.}  \label{r2ex}
\end{table}

Combined with Theorem \ref{mexseqover}, Proposition \ref{reven} gives a combinatorial proof of Corollary \ref{needcomb} for the even $r$ case, as desired.

\section{Summary and future work}
Given the very different nature of $P_e^{>r}(n)$, a subset of regular partitions, and $P_{o,2}^{>r}(n)$, a restricted partition with parts of two colors, it is pleasantly surprising that a single type of overpartition allows bijections to both sets, allowing us to satisfy the request of Kaur, Rana, and Eyyunni for combinatorial proofs of Corollary \ref{needcomb} in a fairly unified way.  This suggests that the overpartitions $\overline{P}_r(n)$ of Definition \ref{opr} provide a more natural combinatorial interpretation of $P_r^{\textrm{mex}}(n)$, the partitions with mex sequence of length at least $r$.  We anticipate further connections between restricted overpartitions and various partition statistics related to the mex.

\end{document}